\documentclass{amsart}

\usepackage{hyperref}
\usepackage{latexsym,amssymb,amsmath,amsfonts}
\usepackage[latin1]{inputenc}
\usepackage{enumerate}
\usepackage[cmtip,arrow]{xy}
\usepackage{pb-diagram,pb-xy}
\usepackage{graphicx}
\usepackage{tikz}
\usepackage{multicol}
\usepackage{url}
\usepackage{listings}
\usepackage{color}

\definecolor{dkgreen}{rgb}{0,0.6,0}
\definecolor{gray}{rgb}{0.5,0.5,0.5}
\definecolor{mauve}{rgb}{0.58,0,0.82}

\usetikzlibrary{matrix,arrows}

\newtheorem{Theorem}{Theorem}[section] 
\newtheorem{Definition}[Theorem]{Definition} 
\newtheorem{Lemma}[Theorem]{Lemma} 
\newtheorem{Corollary}[Theorem]{Corollary} 
 
\newtheorem{Remark}[Theorem]{Remark}

\newtheorem{Proposition}[Theorem]{Proposition} 
\newtheorem{Claim}[Theorem]{Claim}

\begin{document}

\title{Digraphs with at most one trivial critical ideal}
\author{Carlos A. Alfaro \and Carlos E. Valencia \and Adri\'an V\'azquez-\'Avila}
\thanks{The authors were partially supported by SNI}
\date{}
\maketitle

\begin{abstract}
	Critical ideals generalize the critical group, Smith group and the characteristic polynomials of the adjacency and Laplacian matrices of a graph.
	We give a complete characterization of the digraphs with at most one trivial critical ideal.
	Which implies the characterizations of the digraphs whose critical group has one invariant factor equal to one, and the digraphs whose Smith group has one invariant factor equal to one.
\end{abstract}

\section{Introduction}

In the following, the underlying graph of any digraph must be connected, and will contain no loops.
Given a digraph $D=(V,A)$ and a set of indeterminates $X_D=\{x_u \, : \, u\in V(D)\}$, the {\it generalized Laplacian matrix} $L(D,X_D)$ of $D$ is the matrix with rows and columns indexed by the vertices of $D$ given by
\[\small
L(D,X_D)_{u v}=
\begin{cases}
x_u & \text{ if } u=v,\\
-m_{u v} & \text{ otherwise},
\end{cases}
\]
where $m_{u v}$ is the number of arcs going from $u$ to $v$.

\begin{Definition}
For all $1\leq i \leq |V|$, the $i$-{\it th critical ideal} of $D$ is the determinantal ideal given by
\[
I_i(D,X_D)=\langle  \{ {\rm det} (m) \, : \, m \text{ is an }i\times i \text{ submatrix of }L(D,X_D)\}\rangle\subseteq \mathbb{Z}[X_D].
\]
\end{Definition}

We say that a critical ideal $I_i(D,X_D)$ is {\it trivial} when it is equal to $\langle 1 \rangle$.

\begin{Definition}
The {\it algebraic co-rank} $\gamma(D)$ of a digraph $D$ is the number of trivial critical ideals of $D$.
\end{Definition}

Most of the basic properties of the critical ideals were obtained in \cite{corrval}.
For instance, it was proven that if $H$ is an induced subdigraph of $G$, then $I_i(H,X_H)\subseteq I_i(G,X_G)$  for all $i\leq |V(H)|$.
Thus $\gamma(H)\leq \gamma(G)$.
In this work, we are interested in describing the following digraph family:
\begin{Definition}
For $i\in \mathbb{N}$, let
\[\small
\Gamma_{\leq i}=\{D\, :\, D \text{ is a simple connected digraph with } \gamma(D)\leq i\}.
\]
\end{Definition}


Our main result is the characterization of $\Gamma_{\leq 1}$.
This implies two characterizations: the digraphs whose critical group has one invariant factor equal to 1, and the digraphs whose Smith group has one invariant factor equal to one.
These characterizations generalize, in the digraph context, the characterization of the connected graphs with one trivial critical ideal and the characterization of connected graphs with one invariant factor equal to 1. 
For this, we will study in Section~\ref{digraphs} the minimal forbidden digraphs for $\Gamma_{\leq i}$.
These digraphs have been playing a crucial role in the understating the critical ideals and their classification, and allow us to give the characterization of the digraphs with one trivial critical ideal.
In Section~\ref{CharacterizationOfGamma1}, we will give the complete characterization of $\Gamma_{\leq 1}$.
Finally, in Section~\ref{CharacterizationOfInvariantFactors1}, we give the characterizations of the digraphs whose critical group has one invariant factor equal to 1, and the digraphs whose Smith group has one invariant factor equal to one.

\section{Forbidden and $\gamma$-critical digraphs}\label{digraphs}

The major advantage of the critical ideals over the critical group and the Smith group is that critical ideals behave well under induced subdigraph property.
This property allow us to define the following concepts.

\begin{Definition}
We say that a digraph $D$ is \textit{forbidden} for $\Gamma_{\leq k}$ if and only if $\gamma(D)\geq k+1$.
Moreover, let ${\bf Forb}(\Gamma_{\leq k})$ be the set of minimal (under induced subdigraphs property) forbidden digraphs for $\Gamma_{\leq k}$.
\end{Definition}

Given a family $\mathcal F$ of digraphs, a digraph $D$ is called $\mathcal F$-free if no induced subdigraph of $D$ is isomorphic to a member of $\mathcal F$.
Thus $D\in\Gamma_{\leq k}$ if and only if $D$ is ${\bf Forb}(\Gamma_{\leq k})$-free. 
And equivalently, $D$ belongs to $\Gamma_{\geq k+1}$ if and only if $D$ contains a digraph of ${\bf Forb}(\Gamma_{\leq k})$ as an induced subgraph.
Hence a characterization of ${\bf Forb}(\Gamma_{\leq k})$ leads to a characterization of $\Gamma_{\leq k}$.

As $k$ grows, 
it becames difficult to give a complete description of ${\bf Forb}(\Gamma_{\leq k})$.
An alternative technique of computing the elements of ${\bf Forb}(\Gamma_{\leq k})$ is by means of the following definition.
A digraph $D$ is called $\gamma$-\textit{critical} if $\gamma(D\setminus v)< \gamma(D)$ for all $v\in V(D)$.
That is, $D\in {\bf Forb}(\Gamma_{\leq k})$ if and only if $D$ is $\gamma$-critical such that $\gamma(D)\geq k+1$ and $\gamma(D-v)\leq k$.
We implemented this criterion in the software Sage \cite{sage} and Nauty \cite{mckay}, Table \ref{Tab:ForbiddenDigraphs} shows the number of $\gamma$-critical digraphs with at most 6 vertices.
\begin{table}[h!]
	\begin{center}
	\small
	\begin{tabular}{c|cccccccc}
		\hline
		$k \backslash n$ & 2 & 3 & 4 & 5 & 6\\
		\hline
		1 & 2 &  &  &  & \\
		2 & & 7 & 10 &  & \\
		3 & &  & 61 & 1308 & 414\\
		4 & &  &  & 1183 & 542437\\
		5 & &  &  &  & 38229\\
	\end{tabular}
	\end{center}
	\caption{The number of $\gamma$-critical digraphs with $n$ vertices and algebraic co-rank $k$.}
	\label{Tab:ForbiddenDigraphs}
\end{table}

Using the data obtained from this computation, we have that the directed path $\overrightarrow{P_2}$ with 2 vertices and the directed cycle $\overrightarrow{C_2}$ with 2 vertices are the minimal forbidden digraphs for $\Gamma_{\leq 0}$, see Fig.~\ref{fig:digraphs P2andC2}.
Since any other connected digraph with more than $2$ vertices contains $\overrightarrow{P_2}$ or $\overrightarrow{C_2}$ as induced digraphs, then ${\bf Forb}(\Gamma_{\leq 0})=\{\overrightarrow{P_2},\overrightarrow{C_2}\}$.
Therefore, the digraph $T_1$ consisting of an unique vertex is the only $(\overrightarrow{P_2},\overrightarrow{C_2})$-free connected digraph, that is, $\Gamma_{\leq 0}=\{T_1\}$.

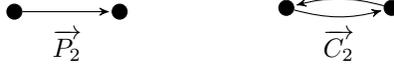
\begin{figure}[h]
\begin{center}
\begin{tabular}{c@{\extracolsep{2cm}}c}
   \begin{tikzpicture}[scale=.7,>=stealth',shorten >=1pt,bend angle = 15]
   \tikzstyle{every node}=[minimum width=0pt, inner sep=2pt, circle]
		\draw (180:1) node (v1) [draw,fill] {};
		\draw (360:1) node (v2) [draw,fill] {};
		\draw[->] (v1) edge (v2);
   \end{tikzpicture}
   
&

   \begin{tikzpicture}[scale=.7,>=stealth',shorten >=1pt,bend angle = 15]
   \tikzstyle{every node}=[minimum width=0pt, inner sep=2pt, circle]
		\draw (180:1) node (v1) [draw,fill] {};
		\draw (360:1) node (v2) [draw,fill] {};
		\draw[->,bend right] (v1) edge (v2);
		\draw[->,bend right] (v2) edge (v1);
   \end{tikzpicture}
   \\
   $\overrightarrow{P_2}$ & $\overrightarrow{C_2}$
\end{tabular}
\end{center}
\label{fig:digraphs P2andC2}
\caption{The $\gamma$-critical digraphs with algebraic co-rank equal to $1$.}
\end{figure}

Also, we have that ${\bf Forb}(\Gamma_{\leq 1})$ consists of $7$ digraphs with $3$ vertices and $10$ digraphs with $4$ vertices (see Figure \ref{fig:ForbDig1}).
These digraphs will be used to completely characterize $\Gamma_{\leq 1}$ in the next section.

Note that the set of digraphs in ${\bf Forb}(\Gamma_{\leq 2})$ is bigger.
For instance, at least contains $61$ digraphs with $4$ vertices, $1308$ digraphs with $5$ vertices, and $414$ digraphs with $6$ vertices.
The size and complexity of ${\bf Forb}(\Gamma_{\leq k})$, for $k\geq 2$, make difficult to obtain a complete characterization of $\Gamma_{\leq k}$.

\section{Digraphs with one trivial critical ideal}\label{CharacterizationOfGamma1}

The main goal of this section is to give the characterization of the digraphs with at most one trivial critical ideal.
As in the case of simple graphs and digraphs with algebraic co-rank equal to zero, the characterization of $\Gamma_{\leq 1}$ relies in that $\Gamma_{\leq 1}$ is closed under induced subdigraphs and that we have previously computed the $\gamma$-critical digraphs with algebraic co-rank equal to $2$.

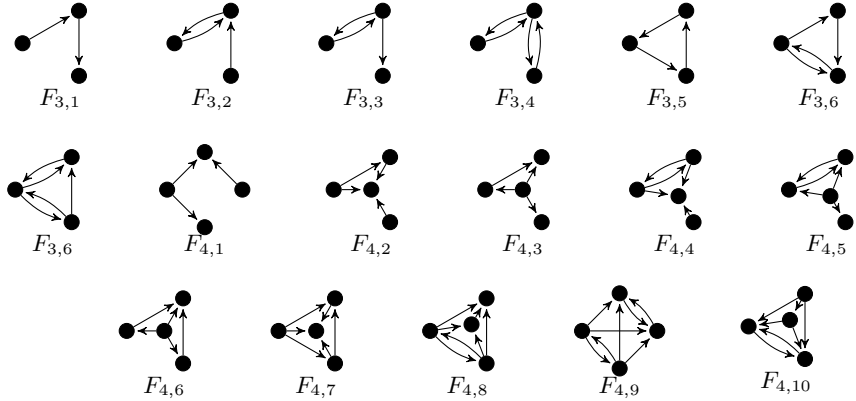
\begin{figure}[h!]
\[
   \begin{tikzpicture}[scale=.5,>=stealth',shorten >=1pt,bend angle = 15]
   \tikzstyle{every node}=[minimum width=0pt, inner sep=2pt, circle]
		\draw (180:1) node (v1) [draw,fill] {};
		\draw (300:1) node (v2) [draw,fill] {};
		\draw (420:1) node (v3) [draw,fill] {};
		\draw (-90:1.5) node (name) {\small $F_{3,1}$};
		\draw[->] (v1) edge (v3);
		\draw[->] (v3) edge (v2);
   \end{tikzpicture}
		\hspace{1cm}
   \begin{tikzpicture}[scale=.5,>=stealth',shorten >=1pt,bend angle = 15]
   \tikzstyle{every node}=[minimum width=0pt, inner sep=2pt, circle]
		\draw (180:1) node (v1) [draw,fill] {};
		\draw (300:1) node (v2) [draw,fill] {};
		\draw (420:1) node (v3) [draw,fill] {};
		\draw (-90:1.5) node (name) {\small $F_{3,2}$};
		\draw[->,bend right] (v1) edge (v3);
		\draw[->] (v2) edge (v3);
		\draw[->,bend right] (v3) edge (v1);
   \end{tikzpicture}
		\hspace{1cm}
   \begin{tikzpicture}[scale=.5,>=stealth',shorten >=1pt,bend angle = 15]
   \tikzstyle{every node}=[minimum width=0pt, inner sep=2pt, circle]
		\draw (180:1) node (v1) [draw,fill] {};
		\draw (300:1) node (v2) [draw,fill] {};
		\draw (420:1) node (v3) [draw,fill] {};
		\draw (-90:1.5) node (name) {\small $F_{3,3}$};
		\draw[->,bend right] (v1) edge (v3);
		\draw[->,bend right] (v3) edge (v1);
		\draw[->] (v3) edge (v2);
   \end{tikzpicture}
		\hspace{1cm}
   \begin{tikzpicture}[scale=.5,>=stealth',shorten >=1pt,bend angle = 15]
   \tikzstyle{every node}=[minimum width=0pt, inner sep=2pt, circle]
		\draw (180:1) node (v1) [draw,fill] {};
		\draw (300:1) node (v2) [draw,fill] {};
		\draw (420:1) node (v3) [draw,fill] {};
		\draw (-90:1.5) node (name) {\small $F_{3,4}$};
		\draw[->,bend right] (v1) edge (v3);
		\draw[->,bend right] (v2) edge (v3);
		\draw[->,bend right] (v3) edge (v1);
		\draw[->,bend right] (v3) edge (v2);
   \end{tikzpicture}
		\hspace{1cm}
   \begin{tikzpicture}[scale=.5,>=stealth',shorten >=1pt,bend angle = 15]
   \tikzstyle{every node}=[minimum width=0pt, inner sep=2pt, circle]
		\draw (180:1) node (v1) [draw,fill] {};
		\draw (300:1) node (v2) [draw,fill] {};
		\draw (420:1) node (v3) [draw,fill] {};
		\draw (-90:1.5) node (name) {\small $F_{3,5}$};
		\draw[->] (v1) edge (v2);
		\draw[->] (v2) edge (v3);
		\draw[->] (v3) edge (v1);
   \end{tikzpicture}
		\hspace{1cm}
   \begin{tikzpicture}[scale=.5,>=stealth',shorten >=1pt,bend angle = 15]
   \tikzstyle{every node}=[minimum width=0pt, inner sep=2pt, circle]
		\draw (180:1) node (v1) [draw,fill] {};
		\draw (300:1) node (v2) [draw,fill] {};
		\draw (420:1) node (v3) [draw,fill] {};
		\draw (-90:1.5) node (name) {\small $F_{3,6}$};
		\draw[->,bend right] (v1) edge (v2);
		\draw[->] (v1) edge (v3);
		\draw[->,bend right] (v2) edge (v1);
		\draw[->] (v3) edge (v2);
   \end{tikzpicture}
		\hspace{1cm}
\]
\[
   \begin{tikzpicture}[scale=.5,>=stealth',shorten >=1pt,bend angle = 15]
   \tikzstyle{every node}=[minimum width=0pt, inner sep=2pt, circle]
		\draw (180:1) node (v1) [draw,fill] {};
		\draw (300:1) node (v2) [draw,fill] {};
		\draw (420:1) node (v3) [draw,fill] {};
		\draw (-90:1.5) node (name) {\small $F_{3,6}$};
		\draw[->,bend right] (v1) edge (v2);
		\draw[->,bend right] (v1) edge (v3);
		\draw[->,bend right] (v2) edge (v1);
		\draw[->] (v2) edge (v3);
		\draw[->,bend right] (v3) edge (v1);
   \end{tikzpicture}
		\hspace{1cm}
   \begin{tikzpicture}[scale=.5,>=stealth',shorten >=1pt,bend angle = 15]
   \tikzstyle{every node}=[minimum width=0pt, inner sep=2pt, circle]
		\draw (180:1) node (v1) [draw,fill] {};
		\draw (270:1) node (v3) [draw,fill] {};
		\draw (360:1) node (v2) [draw,fill] {};
		\draw (450:1) node (v4) [draw,fill] {};
		\draw (-90:1.5) node (name) {\small $F_{4,1}$};
		\draw[->] (v1) edge (v3);
		\draw[->] (v1) edge (v4);
		\draw[->] (v2) edge (v4);
   \end{tikzpicture}
		\hspace{1cm}
   \begin{tikzpicture}[scale=.5,>=stealth',shorten >=1pt,bend angle = 15]
   \tikzstyle{every node}=[minimum width=0pt, inner sep=2pt, circle]
		\draw (180:1) node (v1) [draw,fill] {};
		\draw (300:1) node (v2) [draw,fill] {};
		\draw (420:1) node (v3) [draw,fill] {};
		\draw (0,0) node (v4) [draw,fill] {};
		\draw (-90:1.5) node (name) {\small $F_{4,2}$};
		\draw[->] (v1) edge (v3);
		\draw[->] (v1) edge (v4);
		\draw[->] (v2) edge (v4);
		\draw[->] (v3) edge (v4);
   \end{tikzpicture}
		\hspace{1cm}
   \begin{tikzpicture}[scale=.5,>=stealth',shorten >=1pt,bend angle = 15]
   \tikzstyle{every node}=[minimum width=0pt, inner sep=2pt, circle]
		\draw (180:1) node (v1) [draw,fill] {};
		\draw (300:1) node (v2) [draw,fill] {};
		\draw (420:1) node (v3) [draw,fill] {};
		\draw (0,0) node (v4) [draw,fill] {};
		\draw (-90:1.5) node (name) {\small $F_{4,3}$};
		\draw[->] (v1) edge (v3);
		\draw[->] (v4) edge (v1);
		\draw[->] (v4) edge (v2);
		\draw[->] (v4) edge (v3);
   \end{tikzpicture}
		\hspace{1cm}
   \begin{tikzpicture}[scale=.5,>=stealth',shorten >=1pt,bend angle = 15]
   \tikzstyle{every node}=[minimum width=0pt, inner sep=2pt, circle]
		\draw (180:1) node (v1) [draw,fill] {};
		\draw (300:1) node (v2) [draw,fill] {};
		\draw (420:1) node (v3) [draw,fill] {};
		\draw (300:0.2) node (v4) [draw,fill] {};
		\draw (-90:1.5) node (name) {\small $F_{4,4}$};
		\draw[->,bend right] (v1) edge (v3);
		\draw[->] (v1) edge (v4);
		\draw[->] (v2) edge (v4);
		\draw[->,bend right] (v3) edge (v1);
		\draw[->] (v3) edge (v4);
   \end{tikzpicture}
		\hspace{1cm}
   \begin{tikzpicture}[scale=.5,>=stealth',shorten >=1pt,bend angle = 15]
   \tikzstyle{every node}=[minimum width=0pt, inner sep=2pt, circle]
		\draw (180:1) node (v1) [draw,fill] {};
		\draw (300:1) node (v2) [draw,fill] {};
		\draw (420:1) node (v3) [draw,fill] {};
		\draw (300:0.2) node (v4) [draw,fill] {};
		\draw (-90:1.5) node (name) {\small $F_{4,5}$};
		\draw[->,bend right] (v1) edge (v3);
		\draw[->,bend right] (v3) edge (v1);
		\draw[->] (v4) edge (v1);
		\draw[->] (v4) edge (v2);
		\draw[->] (v4) edge (v3);
   \end{tikzpicture}
		\hspace{1cm}
\]
\[
   \begin{tikzpicture}[scale=.5,>=stealth',shorten >=1pt,bend angle = 15]
   \tikzstyle{every node}=[minimum width=0pt, inner sep=2pt, circle]
		\draw (180:1) node (v1) [draw,fill] {};
		\draw (300:1) node (v2) [draw,fill] {};
		\draw (420:1) node (v3) [draw,fill] {};
		\draw (0,0) node (v4) [draw,fill] {};
		\draw (-90:1.5) node (name) {\small $F_{4,6}$};
		\draw[->] (v1) edge (v3);
		\draw[->] (v2) edge (v3);
		\draw[->] (v4) edge (v1);
		\draw[->] (v4) edge (v2);
		\draw[->] (v4) edge (v3);
   \end{tikzpicture}
		\hspace{1cm}
   \begin{tikzpicture}[scale=.5,>=stealth',shorten >=1pt,bend angle = 15]
   \tikzstyle{every node}=[minimum width=0pt, inner sep=2pt, circle]
		\draw (180:1) node (v1) [draw,fill] {};
		\draw (300:1) node (v2) [draw,fill] {};
		\draw (420:1) node (v3) [draw,fill] {};
		\draw (0,0) node (v4) [draw,fill] {};
		\draw (-90:1.5) node (name) {\small $F_{4,7}$};
		\draw[->] (v1) edge (v2);
		\draw[->] (v1) edge (v3);
		\draw[->] (v1) edge (v4);
		\draw[->] (v2) edge (v3);
		\draw[->] (v2) edge (v4);
		\draw[->] (v3) edge (v4);
   \end{tikzpicture}
		\hspace{1cm}
   \begin{tikzpicture}[scale=.5,>=stealth',shorten >=1pt,bend angle = 15]
   \tikzstyle{every node}=[minimum width=0pt, inner sep=2pt, circle]
		\draw (180:1) node (v1) [draw,fill] {};
		\draw (300:1) node (v2) [draw,fill] {};
		\draw (420:0.2) node (v3) [draw,fill] {};
		\draw (420:1) node (v4) [draw,fill] {};
		\draw (-90:1.5) node (name) {\small $F_{4,8}$};
		\draw[->,bend right] (v1) edge (v2);
		\draw[->] (v1) edge (v3);
		\draw[->] (v1) edge (v4);
		\draw[->,bend right] (v2) edge (v1);
		\draw[->] (v2) edge (v3);
		\draw[->] (v2) edge (v4);
		\draw[->] (v3) edge (v4);
   \end{tikzpicture}
		\hspace{1cm}
   \begin{tikzpicture}[scale=.5,>=stealth',shorten >=1pt,bend angle = 15]
   \tikzstyle{every node}=[minimum width=0pt, inner sep=2pt, circle]
		\draw (180:1) node (v1) [draw,fill] {};
		\draw (270:1) node (v2) [draw,fill] {};
		\draw (360:1) node (v3) [draw,fill] {};
		\draw (450:1) node (v4) [draw,fill] {};
		\draw (-90:1.5) node (name) {\small $F_{4,9}$};
		\draw[->,bend right] (v1) edge (v2);
		\draw[->] (v1) edge (v3);
		\draw[->] (v1) edge (v4);
		\draw[->,bend right] (v2) edge (v1);
		\draw[->] (v2) edge (v3);
		\draw[->] (v2) edge (v4);
		\draw[->,bend right] (v3) edge (v4);
		\draw[->,bend right] (v4) edge (v3);
   \end{tikzpicture}
		\hspace{1cm}
   \begin{tikzpicture}[scale=.5,>=stealth',shorten >=1pt,bend angle = 15]
   \tikzstyle{every node}=[minimum width=0pt, inner sep=2pt, circle]
		\draw (180:1) node (v1) [draw,fill] {};
		\draw (300:1) node (v2) [draw,fill] {};
		\draw (420:1) node (v3) [draw,fill] {};
		\draw (420:0.2) node (v4) [draw,fill] {};
		\draw (-90:1.5) node (name) {\small $F_{4,10}$};
		\draw[->,bend right] (v1) edge (v2);
		\draw[->,bend right] (v2) edge (v1);
		\draw[->] (v3) edge (v1);
		\draw[->] (v3) edge (v2);
		\draw[->] (v3) edge (v4);
		\draw[->] (v4) edge (v1);
		\draw[->] (v4) edge (v2);
   \end{tikzpicture}
\]
\caption{The seventeen $\gamma$-critical digraphs with $3$ and $4$ vertices that have algebraic co-rank equal to $2$.}
\label{fig:ForbDig1}
\end{figure}

Let $\mathfrak F$ be the family of digraphs shown in Figure~\ref{fig:ForbDig1}; consisting of $17$ $\gamma$-critical digraphs with algebraic co-rank equal to $2$.
In appendix \ref{sec:FinForb} there is a sage code to verify the following lemma.

\begin{Lemma}\label{lemma:FinForbGamma}
If $\mathfrak F$ is the family of digraphs shown in Figure \ref{fig:ForbDig1}, then $\mathfrak F\subseteq {\bf Forb}(\Gamma_{\leq 1})$.
\end{Lemma}

A digraph $D$ is {\it complete} if, for every pair $u,v$ of distinct vertices of $D$, both arcs $(uv)$ and $(vu)$ are in $D$.
For simplicity, let $uv$ denote the existence of both arcs $(uv)$ and $(vu)$.
The complete digraph with $n$ vertices is denoted by ${K_n}$.
The trivial digraph with $n$ vertices with no arcs is denoted by $T_n$.
Let $D_1$ and $D_2$ be vertex-disjoint subdigraphs of $D$.
The set of arcs with tails in $V(D_1)$ and heads in $V(D_2)$ is denoted by $(D_1,D_2)_D$.
We say that $(D_1,D_2)_D$ is {\it complete} when is equal to $\{uv\,:\, u\in V(D_1) \text{ and } v\in V(D_2)\}$.

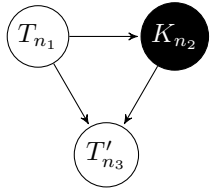
\begin{figure}[h!]
\begin{center}
\begin{tikzpicture}[scale=.7,>=stealth',shorten >=1pt,bend angle = 15]
 	\tikzstyle{every node}=[minimum width=0pt, inner sep=2pt, circle]
 	\draw (30:1.5) node[draw, fill] (1) {\bf\color{white} $K_{n_2}$};
 	\draw (150:1.5) node[draw] (2) {$T_{n_1}$};
 	\draw (270:1.5) node[draw] (3) {$T'_{n_3}$};
 	\draw[->] (1) -- (3);
 	\draw[->] (2) -- (1);
 	\draw[->] (2) -- (3);
 \end{tikzpicture}
\end{center}
\caption{The digraph $\Lambda_{n_1,n_2,n_3}$.}
\label{figure:lambda}
\end{figure}

Let $\Lambda_{n_1,n_2n_3}$ be the digraph defined in the following way:
The vertex set $V(\Lambda_{n_1,n_2n_3})$ is partitioned in three sets $T$, $T'$ and $K$ with $n_1,n_3$ and $n_2$ vertices, respectively, 
such that $T$ and $T'$ are two trivial digraphs, and $K$ is a complete digraph.
Additionally, the arc sets $(T,K)_{\Lambda_{n_1,n_2,n_3}}$, $(T,T')_{\Lambda_{n_1,n_2,n_3}}$ and $({K},T')_{\Lambda_{n_1,n_2,n_3}}$ are complete.
See Figure~\ref{figure:lambda} for a graphical representation of $\Lambda_{n_1,n_2n_3}$.
Let $\Lambda$ be the family of digraphs consisting of all the connected  digraphs $\Lambda_{n_1,n_2n_3}$ for all $n_1,n_2,n_3\geq0$.

\begin{Lemma}\label{lemma:main}
If $\Lambda_{n_1,n_2,n_3}$ is a connected digraph with $n_1+n_2+n_3\geq2$, then $I_1(\Lambda_{n_1,n_2,n_3},\{X_T,Y_K,Z_{T'}\})$ is trivial, and
\[\small
I_2(\Lambda_{n_1,n_2,n_3},\{X_T,Y_K,Z_{T'}\})=
\begin{cases}
\left<\cup_{i=1}^{n_1}x_i, \cup_{i=1}^{n_2}(y_i+1), \cup_{i=1}^{n_3}z_i \right>, & \text{ if } n_1,n_2,n_3\geq 1,\\
\left<x_1 y_1 \right>, & \text{ if } n_1=n_2=1, n_3=0\\
\left<x_1 z_1 \right>, & \text{ if } n_1=n_3=1, n_2=0\\
\left<y_1 z_1 \right>, & \text{ if } n_2=n_3=1, n_1=0\\
\left<\cup_{i=1}^{n_1}x_i \right>, & \text{ if } n_1\geq2,n_2=1,n_3=0,\\
\left<\cup_{i=1}^{n_1}x_i, \cup_{i=1}^{n_2}(y_i+1)\right>, & \text{ if } n_1\geq1,n_2\geq2,n_3=0,\\
\left<y_1y_2-1\right>, & \text{ if } n_1=0,n_2=2,n_3=0,\\
\left<\cup_{i<j}(y_i+1)(y_j+1)\right>, & \text{ if } n_1=0,n_2\geq3,n_3=0,\\
\left<\cup_{i=1}^{n_1}z_i \right>, & \text{ if } n_1=0,n_2=1,n_3\geq2,\\
\left<\cup_{i=1}^{n_2}(y_i+1), \cup_{i=1}^{n_3}z_i \right>, & \text{ if } n_1=0,n_2\geq2,n_3\geq 1,\\
\left<\cup_{i=1}^{n_3}z_i \right>, & \text{ if } n_1=1,n_2=0,n_3\geq 2,\\
\left<\cup_{i=1}^{n_1}x_i \right>, & \text{ if } n_1\geq2,n_2=0,n_3= 1,\\
\left<\cup_{i=1}^{n_1}x_i, \cup_{i=1}^{n_3}z_i \right>, & \text{ if } n_1\geq2,n_2=0,n_3\geq2.\\
\end{cases}
\]
\end{Lemma}

\begin{proof}
	We have that 
	\[
	L(\Lambda_{n_1,n_2,n_3},\{X_T,Y_K,Z_{T'}\})=
	\left[
	\begin{array}{ccc}
	L(T,X_{T}) & -J_{n_1,n_2} & -J_{n_1,n_3} \\
	\bf{0}_{n_2,n_1} & L(K,Y_{K}) & -J_{n_2,n_3} \\
	\bf{0}_{n_3,n_1} & \bf{0}_{n_3,n_2} & L(T',Z_{T'})
	\end{array}
	\right],
	\]
	where $J_{m,n}$ denote the all ones $m\times n$-matrix.
	It is easy to see that the first critical ideal of a digraph with at least one arc is trivial.
	For the case when $n_1,n_2,n_3\geq 1$, we have that the non-singular 2-minors (with positive leading coefficient) of the generalized Laplacian matrix of $\Delta_{n_1,n_2,n_3}$ are of the form:
	$x_ix_j$, $x_i$, $x_iy_j$, $y_i+1$, $y_iz_j$, $x_iz_j$, $z_i$, $y_iy_j-1$, $z_iz_j$.
	Since $x_ix_j$, $x_iy_j$, $y_iz_j$, $x_iz_j$ and $y_iy_j-1$ are in $\left<\cup_{i=1}^{n_1}x_i, \cup_{i=1}^{n_2}(y_i+1), \cup_{i=1}^{n_3}z_i \right>$, we have that $L(\Lambda_{n_1,n_2,n_3},\{X_T,Y_K,Z_{T'}\})=\left<\cup_{i=1}^{n_1}x_i, \cup_{i=1}^{n_2}(y_i+1), \cup_{i=1}^{n_3}z_i \right>$.
	The non-singular 2-minors of the rest of the cases are a subset of the obtained previously, from which the result follows.
\end{proof}

Now we give the characterization of the digraphs with at most one trivial critical ideal.

\begin{Theorem}\label{teo:main}
If $D$ is a connected simple digraph, then the following statements are equivalent:
	\begin{enumerate}[i.]
		\item $D\in \Gamma_{\leq1}$,
		\item $D$ is $\mathfrak F$-free.
		\item $D$ is isomorphic to $\Lambda_{n_1,n_2,n_3}$ for some natural numbers $n_1$, $n_2$ and $n_3$.
	\end{enumerate}
\end{Theorem}

\begin{proof}
By Lemma \ref{lemma:FinForbGamma}, each digraph in $\mathfrak F$ has algebraic co-rank 2, then a digraph in $\Gamma_{\leq1}$ is $\mathfrak F$-free.
Also, Lemma \ref{lemma:main} implies that each digraph in $\Lambda$ belongs to $\Gamma_{\leq1}$.
Therefore, it remains to prove that a $\mathfrak F$-free digraph is isomorphic to $\Lambda_{n_1,n_2,n_3}$ for some integers $n_1$, $n_2$ and $n_3$.
We will proceed by induction on $n=n_1+n_2+n_3$.

\begin{figure}[h!]
\[
   \begin{tikzpicture}[scale=.5,>=stealth',shorten >=1pt,bend angle = 15]
   \tikzstyle{every node}=[minimum width=0pt, inner sep=2pt, circle]
		\draw (180:1) node (v1) [draw,fill] {};
		\draw (300:1) node (v2) [draw,fill] {};
		\draw (420:1) node (v3) [draw,fill] {};
		\draw (-90:1.5) node (name) {\small $A_{1}$};
		\draw[->] (v1) edge (v3);
		\draw[->] (v2) edge (v3);
   \end{tikzpicture}
		\hspace{1cm}
   \begin{tikzpicture}[scale=.5,>=stealth',shorten >=1pt,bend angle = 15]
   \tikzstyle{every node}=[minimum width=0pt, inner sep=2pt, circle]
		\draw (180:1) node (v1) [draw,fill] {};
		\draw (300:1) node (v2) [draw,fill] {};
		\draw (420:1) node (v3) [draw,fill] {};
		\draw (-90:1.5) node (name) {\small $A_{2}$};
		\draw[->] (v3) edge (v1);
		\draw[->] (v3) edge (v2);
   \end{tikzpicture}
		\hspace{1cm}
   \begin{tikzpicture}[scale=.5,>=stealth',shorten >=1pt,bend angle = 15]
   \tikzstyle{every node}=[minimum width=0pt, inner sep=2pt, circle]
		\draw (180:1) node (v1) [draw,fill] {};
		\draw (300:1) node (v2) [draw,fill] {};
		\draw (420:1) node (v3) [draw,fill] {};
		\draw (-90:1.5) node (name) {\small $A_{3}$};
		\draw[->] (v3) edge (v1);
		\draw[->] (v3) edge (v2);
		\draw[->] (v2) edge (v1);
   \end{tikzpicture}
		\hspace{1cm}
   \begin{tikzpicture}[scale=.5,>=stealth',shorten >=1pt,bend angle = 15]
   \tikzstyle{every node}=[minimum width=0pt, inner sep=2pt, circle]
		\draw (180:1) node (v1) [draw,fill] {};
		\draw (300:1) node (v2) [draw,fill] {};
		\draw (420:1) node (v3) [draw,fill] {};
		\draw (-90:1.5) node (name) {\small $A_{4}$};
		\draw[->,bend right] (v1) edge (v3);
		\draw[->,bend right] (v3) edge (v1);
		\draw[->] (v1) edge (v2);
		\draw[->] (v3) edge (v2);
   \end{tikzpicture}
		\hspace{1cm}
   \begin{tikzpicture}[scale=.5,>=stealth',shorten >=1pt,bend angle = 15]
   \tikzstyle{every node}=[minimum width=0pt, inner sep=2pt, circle]
		\draw (180:1) node (v1) [draw,fill] {};
		\draw (300:1) node (v2) [draw,fill] {};
		\draw (420:1) node (v3) [draw,fill] {};
		\draw (-90:1.5) node (name) {\small $A_{5}$};
		\draw[->,bend right] (v1) edge (v3);
		\draw[->,bend right] (v3) edge (v1);
		\draw[->] (v2) edge (v1);
		\draw[->] (v2) edge (v3);
   \end{tikzpicture}
		\hspace{1cm}
   \begin{tikzpicture}[scale=.5,>=stealth',shorten >=1pt,bend angle = 15]
   \tikzstyle{every node}=[minimum width=0pt, inner sep=2pt, circle]
		\draw (180:1) node (v1) [draw,fill] {};
		\draw (300:1) node (v2) [draw,fill] {};
		\draw (420:1) node (v3) [draw,fill] {};
		\draw (-90:1.5) node (name) {\small $A_{6}$};
		\draw[->,bend right] (v1) edge (v2);
		\draw[->,bend right] (v2) edge (v1);
		\draw[->,bend right] (v1) edge (v3);
		\draw[->,bend right] (v3) edge (v1);
		\draw[->,bend right] (v2) edge (v3);
		\draw[->,bend right] (v3) edge (v2);
   \end{tikzpicture}
		\hspace{1cm}
\]
\caption{The digraphs with 3 vertices that have algebraic co-rank equal to $1$.}
\label{fig:AllowedDig1}
\end{figure}
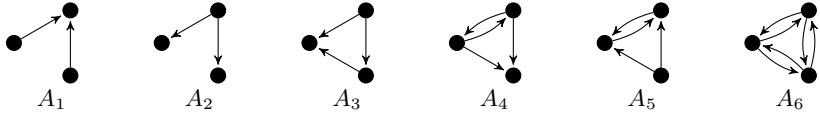

We have that $\overrightarrow{P_2}$ and $\overrightarrow{C_2}$ are the only digraphs with 2 vertices with algebraic co-rank  equal to 1, and in Figure~\ref{fig:AllowedDig1} there are shown the only six digraphs with algebraic co-rank equal to 1.
The result is true for $n\leq3$, since all these digraphs belong to $\Lambda$.

Assume that, for some integer $n\geq 4$, every $\mathfrak F$-free digraph on $n$ vertices is isomorphic to $\Lambda_{n_1,n_2,n_3}$ for some $n_1$, $n_2$ and $n_3$.
Now let $D=(V,A)$ be an $\mathfrak F$-free digraph on $n+1$ vertices, and let $v\in V(D)$.
The digraph $D'=D-v$ is $\mathfrak F$-free, and by inductive hypothesis, $D'$ is isomorphic to $\Lambda_{n_1,n_2,n_3}$ for some integers $n_1$, $n_2$ and $n_3$.
Take three vertices $a,b,c$ in $D'$ such that its underlying graph is connected, and $v$ is adjacent with at least one of the three vertices.
Since $D'$ is isomorphic to $\Lambda$, the induced digraph by the vertices $a,b$ and $c$ is isomorphic to one of the digraphs shown in Figure~\ref{fig:AllowedDig1}. 
We are going to continue with the following claims.

\begin{Claim}\label{claim:complete}
If $D'$ is a complete digraph, then
\begin{enumerate}[i.]
\item each vertex in $D'$ is adjacent to $v$,
\item all arcs between $D'$ and $v$ have the same orientation.
\end{enumerate}
\end{Claim}
\begin{proof}
Since $D$ is connected, then $v$ is adjacent to a vertex $a$ in $D'$.
Take $b\in V(D')$ different from $a$.
Suppose $b$ is not adjacent with $v$, then since there are three possibilities in which $a$ and $v$ are adjacent: either $(av)$, $(va)$, or $av$, that is, both $(av)$ and $(va)$ exist.
However, these possibilities cannot be, since the induced digraph by $a,b$ and $v$ would be isomorphic to $F_{3,2}$, $F_{3,3}$ or $F_{3,4}$.
Thus, each vertex in $D'$ is adjacent to $v$.
Similarly, all arcs between $D'$ and $v$ have the same orientation, otherwise $D'$ would contain an induced subgraph isomorphic to $F_{3,6}$ or $F_{3,7}$.
\end{proof}

\begin{Remark}\label{remark:caminotriangulo}
	Let $a,b,c\in D'$ such that $(ab),(bc)\in A(D')$ and $(ba),(cb)\notin A(D')$.
	If $a$ and $c$ are adjacent, then $(ac)$ is the only possible arc between both vertices.
\end{Remark}

\begin{Claim}\label{claim:casoA1}
	Assume that the unique arcs in $D'[a,b,c]$ are $(ac)$ and $(bc)$.
	Then, $(\{a,b,c\},v)$ is equal to one of the following arcs sets.
	\begin{enumerate}[i.]
		\item $\{(av),(bv)\}$,
		\item $\{(av), (bv), (vc)\}$,
		\item $\{(av), (bv), (cv)\}$,
		\item $\{(av), (bv), vc\}$,
		\item $\{(vc)\}$.
	\end{enumerate}
\end{Claim}
\begin{proof}
	Suppose $(va)$ exists.
	Then, by Remark \ref{remark:caminotriangulo}, $(vc)$ exits.
	There must exists an arc between $b$ and $v$, otherwise $D'[a,b,c,v]$ would be isomorphic to $F_{4,2}$.
	The arcs $vb$ or $(bv)$ cannot exist, otherwise $(ba)$ should exits.
	Finally, it is impossible that the arc $(vb)$ exits, otherwise $D'[a,b,c,v]$ would be isomorphic to $F_{4,6}$.
	Thus, $(va)\notin E(D')$.
	
	Now, suppose $va\in E(D')$.
	By Claim \ref{claim:complete}, $(vc)\in E(D')$.
	And, there should exists an arc between $v$ and $b$, otherwise $D'[a,b,c,v]$ would be isomorphic to $F_{4,4}$.
	But, since $va$ is complete, then by Claim \ref{claim:complete} there is an arc between $v$ and $b$ if and only if there is an arc between $a$ and $b$; which is impossible.
	Then, $va\notin E(D')$.
	
	Suppose $(av)$ exists.
	Neither $(vb)$ or $vb$ exists, otherwise $(ab)$ should exists; which is not possible.
	Considering $(bv)\in E(V')$, the cases $i$, $ii$, $iii$ and $iv$ are yielded.
	
	The cases between $v$ and $b$ can be analyzed similarly.
	Now, we assume there exists no arc between $v$ and the vertices $a$ and $b$.
	There is not possible that $(cv)$ or $cv$ exist, otherwise $(av)$ and $(bv)$ should exist.
	The last case is that $vc\in E(D')$, which corresponds to case $v$.
\end{proof}

\begin{Claim}\label{claim:casoA2}
	Assume that the unique arcs in $D'[a,b,c]$ are $(ca)$ and $(cb)$.
	Then, $(\{a,b,c\},v)$ is equal to one of the following arcs sets.
	\begin{enumerate}[i.]
		\item $\{(va),(vb)\}$,
		\item $\{(va), (vb), (vc)\}$,
		\item $\{(va), (vb), (cv)\}$,
		\item $\{(va), (vb), vc\}$,
		\item $\{(cv)\}$.
	\end{enumerate}
\end{Claim}
\begin{proof}
	Suppose $(va)$ exists.
	There must exists an arc between $v$ and one of the vertices $b$ or $c$, otherwise $D'[a,b,c,v]$ would be isomorphic to $F_{4,1}$.
	Considering $(vb)\in E(V')$, the cases $i$, $ii$, $iii$ and $iv$ are yielded.
	On the other hand, the arcs $vb$ or $(bv)$ cannot exist, otherwise $(ba)$ should exits.
	
	Suppose $va\in E(D')$.
	By Claim \ref{claim:complete}, $(cv)\in E(D')$.
	And, there should exists an arc between $v$ and $b$, otherwise $D'[a,b,c,v]$ would be isomorphic to $F_{4,5}$.
	On the other hand, there is no arc between $v$ and $b$, because there should exist an arc between $a$ and $b$; which is impossible.
	Then, $va\notin E(D')$.
	
	Suppose $(av)$ exists.
	Then, by Remark \ref{remark:caminotriangulo}, $(cv)\in E(D')$.
	There should exists an arc between $v$ and $b$, otherwise $D'[a,b,c,v]$ would be isomorphic to $F_{4,3}$.
	Neither $(vb)$ or $vb$ exists, otherwise $(ba)$ should exists; which is not possible.
	Considering $(bv)\in E(V')$, $D'[a,b,c,v]$ would be isomorphic to $F_{4,6}$; which is impossible.
	Thus, $(av)\notin E(D')$.

	The cases between $v$ and $b$ can be analyzed similarly.
	Now, suppose there exists no arc between $v$ and the vertices $a$ and $b$.
	There is not possible that any of the arcs $(vc)$ or $vc$ exist, otherwise $(va)$ and $(vb)$ should exist.
	Finally, $cv\in E(D')$, which corresponds to case $v$.
\end{proof}

\begin{Claim}\label{claim:casoA3}
	Assume that the unique arcs in $D'[a,b,c]$ are $(ab)$, $(ac)$ and $(bc)$.
	Then, $(\{a,b,c\},v)$ is equal to one of the following arcs sets.
	\begin{enumerate}[i.]
		\item $\{(av), vb, (vc)\}$,
		\item $\{(vb), (vc)\}$,
		\item $\{(av), (bv)\}$.
	\end{enumerate}
\end{Claim}
\begin{proof}
	The arcs $(va)$ or $va$ are not in $E(D')$.
	Since otherwise $(vb)$ and $vc$ should exist, then $D'[a,b,c,v]$ would be isomorphic to $F_{4,7}$ and $F_{4,8}$, respectively.
	Which is impossible.
		
	Suppose $(av)$ exists.
	Then, there must exists an arc between $v$ and one of the vertices $b$ or $c$, otherwise $D'[a,b,c,v]$ would be isomorphic to $F_{4,3}$, which is forbidden.
	The arc $(vb)$ cannot exists, since otherwise it would imply that the arc $(vc)$ should exists, but in this case $D'[a,b,c,v]$ would be isomorphic to $F_{4,7}$; which is not possible.
	In the case when $(bv)\in E(V')$, we have 4 possibilities for $v$ and $c$: $(vc)$, $vc$, $(cv)$ or $v$ is not adjacent with $c$.
	The arcs $(vc)$, $vc$ and $(cv)$ are not possible, since otherwise $D'[a,b,c,v]$ would be isomorphic to $F_{4,7}$, $F_{4,10}$ and $F_{4,7}$, respectively.
	In the case when $v$ is not adjacent with $c$, the arc set $(\{a,b,c\},v)$ corresponds to case $iii$.
	Finally, when the arc $vb$ exists, the arc $(vc)$ exists, and this corresponds to case $i$.
	
	Now, suppose there is no arc between $v$ and $a$.
	If $(vb)\in E(V')$, then $(vc)\in E(V')$ and this case corresponds to $ii$.
	The cases when the arcs $(bv)$ or $vb$ exist, there must exists an arc between $a$ and $v$, which is not the case, therefore these cases are not possible.
	
	Now, suppose there is no arc between $v$ and the vertices $a$ and $b$.
	The three possibilities for the arc between $v$ and $c$ are not possible, since otherwise $F_{4,2}$, $F_{3,1}$ or $F_{3,2}$ would appear.
\end{proof}

\begin{Claim}\label{claim:casoA4}
	Assume that the unique arcs in $D'[a,b,c]$ are $ab$, $(ac)$ and $(bc)$.
	Then, $(\{a,b,c\},v)$ is equal to one of the following arcs sets.
	\begin{enumerate}[i.]
		\item $\{(av), (bv)\}$,
		\item $\{av, bv, (vc)\}$,
		\item $\{(va), (vb), (vc)\}$.
	\end{enumerate}
\end{Claim}
\begin{proof}
	Suppose $(av)\in E(V')$.
	Then, $(bv)$ also is in $E(V')$.
	And there is no arc between $v$ and $c$, otherwise $D'[a,b,c,v]$ would be isomorphic to $F_{4,8}$ or $F_{4,9}$.
	Now, suppose $av\in E(V')$.
	Then, $bv$ and $vc$ are also in $E(V')$, and this case corresponds to $ii$.
	Finally, suppose $(va)\in E(V')$.
	Then, Claim \ref{claim:complete} implies that $(vb)\in E(V')$, and by Remark \ref{remark:caminotriangulo} $(vc)$ also is in $E(V')$.
	This case corresponds to $iii$.
	
	The cases where $v$ and $b$ are adjacent can be analyzed similarly to the previous cases.
	Then, let us continue with the case where there are no arcs between $v$ and the vertices $a$ and $b$.
	That is, when only one of the arcs $(vc)$, $vc$ or $cv$ exists.
	These cases are impossible, since otherwise $D'[a,b,c,v]$ would be isomorphic to $F_{4,4}$, $F_{3,2}$ or $F_{3,1}$, respectively.
\end{proof}

\begin{Claim}\label{claim:casoA5}
	Assume that the unique arcs in $D'[a,b,c]$ are $ab$, $(ca)$ and $(cb)$.
	Then, $(\{a,b,c\},v)$ is equal to one of the following arcs sets.
	\begin{enumerate}[i.]
		\item $\{(va), (vb)\}$,
		\item $\{av, bv, (cv)\}$,
		\item $\{(av), (bv), (cv)\}$.
	\end{enumerate}
\end{Claim}
\begin{proof}
	Note that if $(va)$ or $(vb)$ exists, then by Claim \ref{claim:complete} both arcs must exist.
	And there is no arc between $v$ and $c$, since otherwise $D'[a,b,c,v]$ would be isomorphic to $F_{4,9}$, or $F_{4,10}$.
	Thus, $(\{a,b,c\},v)$ is equal to case $i$.
	If $va$ or $vb$ exists, then by Claim \ref{claim:complete} both arcs must exist.
	And by the same claim, $cv$ is in $E(V')$. 
	Then, $(\{a,b,c\},v)$ is equal to case $ii$.
	Finally, if $(av)$ or $(bv)$ exists, then by Claim \ref{claim:complete} both arcs must exist.
	And by Remark \ref{remark:caminotriangulo} $(cv)$ also is in $E(V')$.
	Therefore, $(\{a,b,c\},v)$ is equal to case $iii$.
	
	Now, let us assume there are no arcs between $v$ and the vertices $a$ and $b$, and only one of the arcs $(vc)$, $vc$ or $cv$ exists.
	The cases when one of the arcs $vc$ or $cv$ exists are impossible, since otherwise $D'[a,b,c,v]$ would be isomorphic to $F_{3,1}$ or $F_{4,5}$, respectively.
	And finally, $vc$ cannot exist, since otherwise by Claim \ref{claim:complete} $v$ should be adjacent to $a$ and $b$.
\end{proof}

By the previous claims, $v$ must be included into one of the three partitions of $D'$: $T$, $T'$ or $K$.
Therefore, $D$ belongs is isomorphic to a digraph in $\Lambda$.
\end{proof}

\section{Applications to the Critical group and the Smith group}\label{CharacterizationOfInvariantFactors1}
Originally, critical ideals were defined as a generalization of the critical group, see \cite{alfaval, alfacorrval, corrval}.
However, it is also a generalization of several other algebraic objects like Smith group or characteristic polynomials of the adjacency and Laplacian matrices.

The characterization of the family of simple connected (di)graphs having critical group with $i$ invariant factors equal to $1$ has been of great interest.
Probably, it was initially posed by R. Cori\footnote{Personal communication with C. Merino}.
However, the first result appeared  when D. Lorenzini noticed in \cite{lorenzini1991} that the graphs having critical group with one invariant factor equal to 1 consist only of the complete graphs. 
After, C. Merino in \cite{merino} posed interest on the characterization of for the cases with $2$ and $3$ invariant factors equal to 1. 
In this sense, few attempts have been done. 
For instance, in \cite{pan} it was characterized the graphs having critical group with 2 invariant factors equal to 1 whose third invariant factor is equal to $n$, $n-1$, $n-2$, or $n-3$. 
In \cite{chan} the characterizations of the same graphs but with a cut vertex, and with number of independent cycles equal to $n-2$ are given.
Recently, a complete characterization of the graphs having critical group with two invariant factors equal to 1 was obtained in \cite{alfaval}.

In this section, we will focus in giving a characterization of the digraphs whose critical group has one invariant factor equal to 1, and the characterization of the digraphs whose Smith group has one invariant factor equal to one.

Let us recall the concepts of adjacency matrix, Laplacian matrix, Smith group and critical group.
The {\it adjacency matrix} $A(D)$ and {\it Laplacian matrix} $L(D)$ of $D$ are the evaluation of $-L(D,X_D)$ and $L(D,X_D)$ at $X={\bf 0}$ and at $X=\deg^+(D)$, respectively, where $\deg^+(D)$ is the out-degree vector of $D$.
By considering a $m\times n$ matrix $M$ with integer entries as a linear map $M:\mathbb{Z}^n\rightarrow \mathbb{Z}^m$, the {\it cokernel} of $M$ 
is the quotient module $\mathbb{Z}^{n}/{\rm Im}\, M$. 
This finitely generated abelian group becomes a graph invariant when we take the matrix $M$ to be the adjacency or Laplacian matrix of the graph.

The cokernel of $A(D)$ is known as the {\it Smith group} of the digraph $D$ and is denoted $S(D)$, and the torsion part of the cokernel of $L(D)$ is knwon as the {\it critical group} $K(D)$ of $D$.
For a survey on Smith groups see \cite{rushanan}.
The critical group is especially interesting; for connected graphs since its order is equal to the number of spanning trees of the graph.
The critical group has been studied intensively over the last 30 years on several contexts: the {\it group of components} \cite{lorenzini1991,lorenzini2008}, the {\it Picard group} \cite{bhn,biggs1999}, the {\it Jacobian group} \cite{bhn,biggs1999}, the {\it sandpile group} \cite{cori},  {\it chip-firing game} \cite{biggs1999,merino}, or {\it Laplacian unimodular equivalence} \cite{gmw,merris}.

One way to compute the cokernel of $M$ is by means of the Smith normal form and its associated invariant factors, see \cite[Theorem 1.4]{lorenzini1989}.
Thus, the cokernel of $M$ can be described as:
$coker(M)\cong \mathbb{Z}_{f_1} \oplus \mathbb{Z}_{f_2} \oplus \cdots \oplus\mathbb{Z}_{f_{r}}$,
where $f_1, f_2, ..., f_{r}$ are positive integers with $f_i \mid f_j$ for all $i\leq j$.
These integers are called {\it invariant factors} of $M$.
We might refer the reader to the Stanley's survey  \cite{stanley} on the Smith normal forms in combinatorics for more details in the topic.

Computation of the Smith normal form of the adjacency or Laplacian matrix is a standard technique to determine the Smith or critical group of a graph. 
It is well known that this can be achieved through integral row and column operations.
One more way (see \cite{alfaval0,DGW}) of determining the structure is by working directly within the particular abelian group to identify a cyclic decomposition for it. 
For it, we need to know the orders of the group, exhibit a set of elements and show that their orders divide the orders of the cyclic
factors, and finally, show these elements do indeed generate the group.

An alternative in computing the invariant factors is by means of the following formula (see \cite{jacobson}). 
If $\Delta_i(M)$ is the {\it greatest common divisor} of the $i$-minors of the matrix $M$, then the $i$-{\it th} invariant factor $f_i$ is equal to $\Delta_i(M)/ \Delta_{i-1}(M)$, where $\Delta_0(M)=1$.
The following result uses this formula to build a bridge between the critical groups and critical ideals.

\begin{Proposition}\cite{corrval}\label{teo:eval}
If $\deg^+(D)=(\deg^+_D(v_1), ..., \deg^+_D(v_n))$ is the out--degree vector of $D$, and $f_1\mid\cdots\mid f_{n-1}$ are the invariant factors of $K(D)$, then
\[
I_i(D,X_D)|_{X_D=\deg^+(D)}=\left< \prod_{j=1}^{i} f_j \right> =\left< \Delta_i(L(D))\right>\text{ for all }1\leq i\leq n-1.
\]
\end{Proposition}

Thus, if the critical ideal $I_i(D,X_D)$ is trivial, then $\Delta_i(L(D))$ and $f_i$ are equal to $1$.
Equivalently, if $\Delta_i(L(D))$ and $f_i$ are not equal to $1$, then the critical ideal $I_i(D, X_D)$ is not trivial.
The counterpart to the algebraic co-rank, in the critical group, is the number $f_1(D)$ of invariant factors equal to 1.
\[\small
\mathcal{G}_i=\{D\, :\, D \text{ is a simple connected digraph with } f_1(D)=i\}.
\]
Moreover, $\mathcal{G}_i\subseteq \Gamma_{\leq i}$ for all $i\geq 0$.
By Proposition \ref{teo:eval}, to obtain the classification of the digraphs whose critical group has exactly one invariant factor equal to one, we only need to evaluate the out-degree of the vertices in the second critical ideal corresponding to each digraph in Lemma \ref{lemma:main}.

\begin{Corollary}
The critical group of a connected digraph has exactly one invariant factor equal to 1 if and only if is isomorphic to the digraph $\Lambda_{n_1,n_2,n_3}$ where $n_1,n_2,n_3$ satisfy one of the following conditions.
\begin{multicols}{2}\small
\begin{itemize}
\item $n_1,n_2,n_3\geq1$,
\item $n_1,=n_2=1$, $n_3=0$,
\item $n_1,=n_3=1$, $n_2=0$,
\item $n_2,=n_3=1$, $n_1=0$,
\item $n_1\geq0$, $n_2\geq2$, $n_3\geq0$,
\item $n_1=0$, $n_2=1$, $n_3\geq2$,
\item $n_1=1$, $n_2=0$, $n_3\geq2$,
\item $n_1\geq2$, $n_2=0$, $n_3\geq2$.
\end{itemize}
\end{multicols}
\end{Corollary}
\begin{proof}
	We have that the out-degree of a vertex $v$ in $\Lambda_{n_1,n_2,n_3}$ is as follows.
	\[
		\deg^+(v)=
		\begin{cases}
			n_2+n_3 & \text{if } v\in T_{n_1},\\
			n_2+n_3-1 & \text{if } v\in K_{n_2},\\
			0 & \text{if } v\in T'_{n_3}.\\
		\end{cases}
	\]
	By evaluating each indeterminate at its associated out-degree in the second critical ideal of Lemma \ref{lemma:main}, we obtain that the only cases that the ideal trivializes are $n_1\geq2,n_2=1,n_3=0$ and $n_1\geq2,n_2=0,n_3=1$.
\end{proof}

Similar arguments used in \cite{corrval} give us the following result for the Smith group.

\begin{Proposition}\label{teo:eval1}
If $f_1\mid\cdots\mid f_{n-1}$ are the invariant factors of $S(D)$, then
\[
I_i(D,X_D)|_{X_D=\bf 0}=\left< \prod_{j=1}^{i} f_j \right> =\left< \Delta_i(A(D))\right>\text{ for all }1\leq i\leq n-1.
\]
\end{Proposition}

By evaluating each indeterminate at zero in the second critical ideal of Lemma \ref{lemma:main}, we obtain the characterization.

\begin{Corollary}
The Smith group of a connected digraph has exactly one invariant factor equal to 1 if and only if is isomorphic to the digraph $\Lambda_{n_1,n_2,n_3}$ where $n_1,n_2,n_3$ satisfy one of the following conditions.
\begin{multicols}{2}\small
\begin{itemize}
\item $n_1,=n_2=1$, $n_3=0$,
\item $n_1,=n_3=1$, $n_2=0$,
\item $n_2,=n_3=1$, $n_1=0$,
\item $n_1\geq2$, $n_2=1$, $n_3=0$,
\item $n_1=0$, $n_2=1$, $n_3\geq2$,
\item $n_1=1$, $n_2=0$, $n_3\geq2$,
\item $n_1\geq2$, $n_2=0$, $n_3=1$,
\item $n_1\geq2$, $n_2=0$, $n_3\geq2$.
\end{itemize}
\end{multicols}
\end{Corollary}

\appendix
\section{SAGE code for verifying $\mathfrak F\subseteq {\mathbf{ Forb}}( \Gamma_{ \leq 1 } )$}\label{sec:FinForb}

\begin{lstlisting}
# Forbidden digraphs with 3 vertices
F3 = [];
F3.append(DiGraph({0:[2],2:[1]}, name='F31'));
F3.append(DiGraph({0:[2],1:[2],2:[0]}, name='F32'));
F3.append(DiGraph({0:[2],2:[0,1]}, name='F33'));
F3.append(DiGraph({0:[2],1:[2],2:[0,1]}, name='F34'));
F3.append(DiGraph({0:[1],1:[2],2:[0]}, name='F35'));
F3.append(DiGraph({0:[1,2],1:[0],2:[1]}, name='F36'));
F3.append(DiGraph({0:[1,2],1:[0,2],2:[0]}, name='F37'));

# Forbidden digraphs with 4 vertices
F4 = [];
F4.append(DiGraph({0:[2,3],1:[3]}, name='F41'));
F4.append(DiGraph({0:[2,3],1:[3],2:[3]}, name='F42'));
F4.append(DiGraph({0:[2],3:[0,1,2]}, name='F43'));
F4.append(DiGraph({0:[2,3],1:[3],2:[0,3]}, name='F44'));
F4.append(DiGraph({0:[2],1:[2],3:[0,1,2]}, name='F45'));
F4.append(DiGraph({0:[2],2:[0],3:[0,1,2]}, name='F46'));
F4.append(DiGraph({0:[1,2,3],1:[2,3],2:[3]}, name='F47'));
F4.append(DiGraph({0:[1,2,3],1:[0,2,3],2:[3]}, name='F48'));
F4.append(DiGraph({0:[1,2,3],1:[0,2,3],2:[3],3:[2]}, name='F49'));
F4.append(DiGraph({0:[1],1:[0],2:[0,1,3],3:[0,1]}, name='F410'));

def Gamma(L,n):
	for i in range(n+1):
		I = R.ideal(L.minors(i))
		if( I.is_one() == True ):
			next
		else:
			return i-1
	return n

def isForb(L,n):
	for i in range(n):
		S = range(n)
		S.remove(i)
		I = R.ideal(L[S,S].minors(2))
		if( I.is_one() == True ):
			return "not forbidden"
		else:
			next
	return "forbidden"

for F in F3:
	n = len(F)
	R = macaulay2.ring("ZZ","[x0,x1,x2]").to_sage();
	R.inject_variables();
	Laplacian = diagonal_matrix(list(R.gens())) - F.adjacency_matrix()
	F.show()
	print('Forbidden graph ' + F.name() + ' has Gamma = ' + str(Gamma(Laplacian,n)) + ', and is ' + isForb(Laplacian,n) + '\n')

for F in F4:
	n = len(F)
	R = macaulay2.ring("ZZ","[x0,x1,x2,x3]").to_sage();
	R.inject_variables();
	Laplacian = diagonal_matrix(list(R.gens())) - F.adjacency_matrix()
	F.show()
	print('Forbidden graph ' + F.name() + ' has Gamma = ' + str(Gamma(Laplacian,n)) + ', and is ' + isForb(Laplacian,n) + '\n')
\end{lstlisting}


\vspace{.5cm}

{\parindent=0pt
\obeylines
Banco de M\'exico, 
Mexico City, Mexico
corresponding author: carlos.alfaro@banxico.org.mx, alfaromontufar@gmail.com
\par}
\vspace{0.5cm}
{\parindent=0pt
\obeylines
Departamento de Matem\'aticas, 
Centro de Investigaci\'on y de Estudios Avanzados del IPN, 
Apartado Postal 14-740, 07000 Ciudad de M\'exico, Mexico
cvalencia@math.cinvestav.edu.mx
\par}
\vspace{0.5cm}
{\parindent=0pt
\obeylines
Subdirecci\'on de Ingenier\'ia y Posgrado
Universidad Aeron\'autica en Quer\'etaro
adrian.vazquez@unaq.edu.mx
\par}


\begin{thebibliography}{99} 
\bibitem{alfaval0}{C.A. Alfaro and C.E. Valencia, On the sandpile group of the cone of a graph, Linear Algebra and Its Applications 436 (2012) 1154--1176.}
\bibitem{alfaval}{C.A. Alfaro and C.E. Valencia, Graphs with two trivial critical ideals, Discrete Applied Mathematics 167 (2014) 33--44.}
\bibitem{alfacorrval}{C.A. Alfaro, H.H. Corrales and C.E. Valencia, Critical ideals of signed graphs with twin vertices, Advances in Applied Mathematics 86 (2017) 99--131.}
\bibitem{bhn}{R. Bacher, P. de la Harpe and T. Nagnibeda, The lattice of integral flows and the lattice of integral cuts on a finite graph, Bull. Soc. Math. France 125 (1997) 167--198.}
\bibitem{biggs1999}{N. Biggs, Chip-firing and the critical group of a graph, J. Alg. Combin. 9 (1999) 25--46.}
\bibitem{chan}{W.H. Chan, Y. Hou and W.C. Shiu, Graphs whose critical groups have larger rank, Acta Math. Sinica 27 (2011) 1663--1670.}
\bibitem{cori}{R. Cori and D. Rossin, On the sandpile group of dual graphs, European J. Combin. 21 (4) (2000) 447--459.}
\bibitem{corrval}{H.H. Corrales and C.E. Valencia, On the critical ideals, Linear Algebra and its Applications 439 (2013) 3870--3892.}
\bibitem{DGW}{J.E. Ducey, J. Gerhard and N. Watson, The Smith and Critical Groups of the Square Rook's Graph and its Complement, Electron. J. Combin. 23 (2016) 4 paper 4.9.}
\bibitem{gmw}{R. Grone, R. Merris and W. Watkins, Laplacian unimodular equivalence of graphs. In: R. Brualdi, S. Friedland and V. Klee (Eds.) Combinatorial and Graph-Theoretical Problems in Linear Algebra, Springer-Verlag (1993) 175--180.}
\bibitem{jacobson}{N. Jacobson, Basic Algebra I, Second Edition, W. H. Freeman and Company, New York, 1985.}
\bibitem{lorenzini1989}{D.J. Lorenzini, Arithmetical Graphs, Math. Ann. 285 (1989) 481--501.}
\bibitem{lorenzini1991}{D.J. Lorenzini, A finite group attached to the laplacian of a graph, Discrete Math. 91 (1991) 277--282.}
\bibitem{lorenzini2008}{D.J. Lorenzini, Smith normal form and Laplacians, J. Combin. Theory B 98 (2008) 1271--1300.}
\bibitem{mckay}{B.D. McKay, {\it Nauty User's Guide (Version 2.4)}, available at http://cs.anu.edu.au/$\sim$bdm/nauty/.}
\bibitem{merino}{C. Merino, The chip-firing game, Discrete Math. 302 (2005) 188--210.}
\bibitem{merris}{R. Merris, Unimodular Equivalence of Graphs, Linear Algebra Appl. 173 (1992),181-189}
\bibitem{pan}{Y. Pan and J. Wang, A note on the third invariant factor of the Laplacian matrix of a graph, Journal of University of Science and Technology of China (2011) 6 471--476.}
\bibitem{rushanan}{J.J. Rushanan, Topics in integral matrices and abelian group codes. ProQuest LLC, Ann Arbor, MI, 1986. Thesis (Ph.D.) - California Institute of Technology.}
\bibitem{stanley}{R.P. Stanley, Smith normal form in combinatorics, J. Combin. Theory A 144 (2016) 476--495.}
\bibitem{sage}{W. Stein. {\it Sage: Open Source Mathematical Software (Version 3.0)}. The
Sage Group, 2008. http://www.sagemath.org.}
\end{thebibliography}
\end{document}